\documentclass[11pt,a4paper]{amsart}
\usepackage {mypackage}
\date{\today}

\newcolumntype{C}[1]{>{\centering\arraybackslash$}m{#1}<{$}}

\author{Nicolas Tholozan}
\address{\'ENS -- PSL, CNRS \\
45 rue d'Ulm\\
75005 Paris}
\email{nicolas.tholozan@ens.psl.eu}

\dedicatory{To Fran\c cois Labourie, for his $60^{\textit{th}}+\epsilon$ birthday.}

\title[Cohomological invariants of representation varieties]{Chern--{S}imons theory and cohomological invariants of representation varieties}

\begin{document}

\begin{abstract}
We prove a general local rigidity theorem for pull-backs of homogeneous forms on reductive symmetric spaces under representations of discrete groups. One application of the theorem is that the volume of a closed manifold locally modelled on a reductive homogeneous space $G/H$ is constant under deformation of the $G/H$-structure. The proof elaborates on an argument given by Labourie for closed anti-de Sitter $3$-manifolds.

The core of the work is a reinterpretation of old results of Cartan, Chevalley and Borel, showing that the algebra of $G$-invariant forms on $G/H$ is generated by ``Chern--Weil forms'' and ``Chern--Simons forms''.
\end{abstract}

\maketitle

\tableofcontents

\section*{Introduction}

Throughout the paper, we consider $G$ a connected real semisimple Lie group with finite center, $\sigma$ an involutive automorphism of $G$ and $H$ the subgroup of $G$ fixed by $\sigma$. The right quotient $G/H$ is called a \emph{reductive symmetric space}. The involution $\sigma$ is called a \emph{Cartan involution} when $H$ is compact. In that case, $H$ is a maximal compact subgroup of $G$ and $G/H$ is a Riemannian symmetric space of non-compact type sometimes called \emph{the} symmetric space of $G$. By a result of E. Cartan, every $G$-invariant differential form on $G/H$ is closed, hence the cohomology of the complex of $G$-invariant forms on $G/H$ with values in $\C$ is isomorphic to the algebra of $G$-invariant forms, denoted $\Omega^\bullet_\inv(G/H,\C)$.

We also fix a smooth connected manifold $M$, with universal cover $\tilde M$ and fundamental group $\pi_1(M)$. Let $\rho: \pi_1(M) \to G$ be a representation and $\tilde s: \tilde M \to G/H$ a smooth $\rho$-equivariant map. This map factors to a smooth section $s$ of the flat $G/H$-bundle
\[M\times_\rho (G/H)\]
and the pull-back by $\tilde s$ of any $\omega\in \Omega^\bullet_\inv(G/H,\C)$ factors to a closed form $s^*\omega$ on $M$ (see Section \ref{ss: Flat bundles}). An application of Stokes' formula shows that the de Rahm cohomology class $[s^*\omega]$ of $s^*\omega$ only depends on the homotopy class of $s$. In particular, since the Riemannian symmetric space of $G$ is contractible, all the sections of the bundle $M\times_\rho (G/H)$ are homotopic. Thus, in that case, $[s^*\omega]$ only depends on the representation $\rho$, and we denote it by $\rho^*\omega$.\\

\subsection{Main results}

The purpose of this paper is to prove that the cohomology classes $[s^* \omega]$ do not vary when the representation $\rho$ moves in the representation variety $\Hom(\pi_1(M), G)$. To be more precise, let $(\rho_t)_{t\in [0,1]}$ be a smooth path in $\Hom(\pi_1(M),G)$. Any smooth section $s_0$ of the flat bundle $M\times_{\rho_0}(G/H)$ extends to a smooth family of sections $s_t$ of $M\times_{\rho_t}(G/H)$, unique up to homotopy (see Section \ref{ss: Families of bundles}).

\begin{MonThm} \label{t: Main Theorem}
Let $\omega$ be a $G$-invariant form on $G/H$. For any smooth family $(\rho_t)_{t\in [0,1]}$ of representations of $\pi_1(M)$ into $G$ and any smooth family $(s_t)_{t\in [0,1]}$ of sections of $M\times_{\rho_t} (G/H)$, we have
\[ [s_t^*\omega] = [s_0^*\omega]\]
for all $t\in [0,1]$.
\end{MonThm}

The general formulation of this theorem is meant to encompass two situations: the one where $K$ is a maximal compact subgroup (in which case the choice of $s$ is irrelevant) and the one where $s$ is a local diffeomorphism and $\omega$ a $G$-invariant volume form. We now detail the consequences of Theorem~\ref{t: Main Theorem} in these two contexts.

\subsubsection{Pull-back of group cohomology}

In this section, we take $H = K$ to be a maximal compact subgroup of $G$. The \emph{Van Est isomorphism} $\VE$ identifies the algebra of $G$-invariant forms on $G/K$ to the \emph{continuous cohomology} $\HH^\bullet_c(G,\C)$ (see Section \ref{ss: Continuous group cohomology Proof }).

Approaching the classifying space of any finitely presented group $\Gamma$ by closed manifolds, we get from Theorem \ref{t: Main Theorem} the following corollary:

\begin{MonCoro} \label{c: Main Corollary Cohomology}
Let $\Gamma$ be a finitely presented group. Then, for every $\omega \in \HH^\bullet_c(G,\C)$, the map 
\begin{eqnarray*}
\Hom(\Gamma,G) &\to & \HH^\bullet(\Gamma, \C) \\
\rho & \mapsto & \rho^*\omega 
\end{eqnarray*}
is locally constant on $\Hom(\Gamma,G)$. 
\end{MonCoro}

The cohomology classes $\rho^*\omega$ define topological invariants that are constant on connected components of $\Hom(\Gamma, G)$ and can sometimes help distinguish these components. A famous example of this phenomenon is when $\Gamma$ is the fundamental group of a closed oriented surface $S$ of genus $g\geq 2$ and $G= \PSL(2,\R)$. For a suitably chosen generator $\omega$ of $\HH^\bullet_c(\PSL(2,\R), \C) \simeq \C$, Milnor \cite{Milnor58} proved that the number
\[\int_S \rho^*\omega\]
is an integer contained in the interval $[2-2g, 2g-2]$, called the \emph{Euler class} of the representation $\rho$. Moreover, Goldman \cite{Goldman88} and Hitchin \cite{Hitchin87} proved that the Euler class classifies the connected components of $\Hom(\Gamma, \PSL(2,\R))$.

\subsubsection{Volume of locally homogeneous manifolds}

In this section, $G$ is a semisimple Lie group and $H$ any reductive subgroup. Let $M$ be an orientable manifold of the same dimension as $G/H$. A \emph{$G/H$-structure} on $M$ can be defined as the data of a pair $(\dev, \rho)$ where $\rho:\pi_1(M) \to G$ is a homomorphism called the \emph{holonomy} of the $G/H$-structure and $\dev: \tilde M \to G/H$ is a local diffeomorphism called the \emph{developing map}. 

The hypotheses on $G$ and $H$ imply the existence of a $G$-invariant volume forme $\vol_{G/H}$ on the homogeneous space $G/H$, whose pull-back by the developping factors to a volume form $\dev^*\vol_{G/H}$ on $M$. The \emph{volume} of the $G/H$-structure is by definition the integral of this volume form over $M$
\[\Vol(M, \dev) = \int_M \dev^* \vol_{G/H}~.\]

The Ehresmann--Thurston principle (see \cite{BergeronGelander04}) roughly states that every deformation of the holonomy representation corresponds (in an essentially unique way) to a deformation of the $G/H$-structure of $M$. Theorem \ref{t: Main Theorem} implies that the volume is constant along such deformations.

\begin{MonCoro} \label{c: Main Corollary Volume}
Let $G$ be a semisimple Lie group with finite center and finitely many components and $H$ a reductive subgroup of $G$. Let $M$ be a closed manifold of the same dimension as $G/H$ and $(\dev_t, \rho_t)_{t\in [0,1]}$ a continuous family of $G/H$-structures on $M$. Then 
\[\Vol(M, \dev) = \Vol(M, \dev_0)\]
for all $t\in [0,1]$.
\end{MonCoro}

\subsection{Earlier results}

Theorem \ref{t: Main Theorem} contains as a particular case many topological invariants that have been extensively studied before, such as Toledo invariants of surface group representations (or more generally K\"ahler groups) into Lie groups of Hermitian type \cite{BIW10,Toledo89}, volume of representations of hyperbolic lattices \cite{BBI13,Francaviglia04,KimKim13,BCG07}, complex volume of representations of $3$-manifold groups \cite{GTZ15,BFG14}, or volume of compact quotients of reductive homogeneous spaces. To our knowledge, in all these examples, Theorem \ref{t: Main Theorem} is already known, and our intention here is merely to give a (almost) unified presentation of these different results. Let us discuss some of these applications.

\subsubsection{Toledo invariant of surface group representations}

Let $S$ be a closed surface. The Toledo invariant generalizes the Euler class to representations of $\pi_1(S)$ into a simple Lie group $G$ of Hermitian type. In that case, the cohomology group $\HH^2_c(G,\C)$ is generated by a class $\omega$ (corresponding to a $G$-invariant K\"ahler form on the symmetric space). Given a representation $\rho: \pi_1(S)\to G$, the Toledo invariant of $\rho$ is defined by
\[\tau(\rho) = \int_S \rho^*\omega~.\] 
It is rational, satisfies a Milnor--Wood inequality, and representations with maximal Toledo invariant have strong geometric properties \cite{Toledo89,BIW10}. Unlike in the case of $\PSL(2,\R)$ however, the Toledo invariant does not always distinguish connected components of $\Hom(\Gamma, G)$ (see \cite{BGG06}). The local rigidity of the Toledo invariant (and, in fact, its rationality) is a well-known consequences of Chern--Weil theory: one shows that the Toledo invariant is the degree of the pull-back of a complex automorphic line bundle on $G/K$.

\subsubsection{Volume of representations into $\Isom_+(\H^d)$}

Another well-studied generalization of the Euler class is when $M$ is a closed $d$-manifold and $G= \SO_\circ(d,1)$ is the group of orientation preserving isometries of the hyperbolic $d$-space $\H^d$. The continuous cohomology of $\SO_\circ(d,1)$ is generated in degree $d$ by a class $\omega$ corresponding to the volume form of the hyperbolic $d$-space. Given $\rho: \pi_1(M) \to\SO_\circ(d,1)$, the number $\int_M \rho^*\omega$ is called the  \emph{volume} of the representation $\rho$. 

For even $d$, the volume of $\rho$ is (up to a universal constant) the Euler class of the pull-back of the tangent bundle $T \H^d$ and is thus an integer. In contrast, for odd $d$, the volume cannot be directly related to a characteristic class, and its values remain mysterious. Besson--Courtois--Gallot proved in \cite{BCG07} that the volume is locally constant on $\Hom(\pi_1(M), \SO_\circ (d,1))$, using Schl\"afli's formula for the variation of the volume of a simplex. This proof seems quite specific to hyperbolic geometry. 

An alternative proof in dimension $3$ follows from the identification of the volume of $\rho$ with the imaginary part of its \emph{Cheeger--Chern--Simons invariant}. While this proof was probably known to experts before Besson--Courtois--Gallot's work, it seems it was only written later in \cite{GTZ15}. 

Theorem \ref{t: Main Theorem} gives in particular an alternative proof of Besson--Courtois--Gallot's theorem which, in odd dimension, is based on Chern--Simons theory and is strongly related to \cite{GTZ15}.

\subsubsection{Complex volumes of representations of $3$-manifold groups}

In \cite{GTZ15}, Garoufalidis, Thurston and Zickert generalize the Cheeger--Chern--Simons invariant to representations of the fundamental group of a closed $3$-manifold $M$ into $\SL(d,\C)$ (and, in fact, any simply connected simple complex Lie group).

In that case, the group $G$ itself carries a holomorphic bi-invariant $3$-form~$\omega$, given on the Lie algebra $\g$ by 
\[\omega(u,v,w)= \Tr(u [v,w])~.\]
Let $\rho$ be a representation of $\pi_1(M)$ into $\SL(d,\C)$. Since $\SL(d,\C)$ is connected and simply connected and $M$ has dimension $3$, the flat bundle $M\times_\rho \SL(d,\C)$ always admits a section $s$. Moreover, given two such sections $s_1$ and $s_2$, the difference $\int_M s_1^* \omega - \int_M s_2^*\omega$ is an integral multiple of $4\pi^2$. The \emph{complex volume} of $\rho$ is then defined by 
\[\Vol_\C(\rho) = - i \int_M s^*\omega\in \C/4\pi^2i\Z~.\]
When $n=3$, the real part of $\Vol_\C(\rho)$ is the volume of $\rho$ defined above.

This complex volume can be interpreted as a Chern--Simons class. More precisely, it is the Cheeger--Chern--Simons invariant of the flat connection of monodromy $\rho$ on the trivial complex vector bundle of rank $n$. General arguments of Chern--Simons theory then imply that it is locally constant on $\Hom(\pi_1(M),\SL(n,\C))$. The present work can be viewed as a generalisation of this fact.\\

\begin{rmk}
Both the Toledo invariant, the volume of representions in $\Isom_+(\H^d)$ and the complex volume of representations of $3$-manifold groups have also been defined and investigated of manifolds $M$ with boundary (see \cite{BIW10,BBI13,KimKim13,Francaviglia04,GTZ15,BFG14}).

While some local rigidity theorems have been proven in higher dimension, when $M$ has dimension $2$ or $3$ the Toledo invariant or complex volume are typically not locally constant on character varieties.  It leads to a very different story that we do not consider further here.
\end{rmk}

\subsubsection{Volume anti-de Sitter $3$-manifolds}

Despite its proximity with the local rigidity results mentioned above, Corollary \ref{c: Main Corollary Volume} on the volume of $G/H$-manifolds did not seem to be known before. In the first interesting case where $G/H$ is the \emph{anti-de Sitter space of dimension $3$}, the question whether the volume is invariant under continuous deformations was asked in the influential survey \cite[Question 2.3]{QuestionsAdS}.\\

Recall that the anti-de Sitter $3$-space $\AdS^3$ can be seen as the Lie group $\PSL(2,\R)$ equipped with its Killing metric. Its isometry group is (up to finite index) the group $\PSL(2,\R)\times \PSL(2,\R)$ acting by left and right multiplication. Anti-de Sitter structures on $3$-manifolds have been extensively studied \cite{KulkarniRaymond85,Goldman85,Salein00,KasselThese,Tholozan3} and are now well-understood. In particular, Kassel (partly relying on previous works) proved in \cite{KasselThese} that, up to finite covers, closed anti-de Sitter $3$-manifolds have the form
\[(j, \rho)(\pi_1(S))\backslash \PSL(2,\R)\]
where $S$ is a closed hyperbolic surface, $j:\pi_1(S)\to \PSL(2,\R)$ is the holonomy of the hyperbolic metric on $S$, and $\rho: \pi_1(S)\to \PSL(2,\R)$ is another representation such that there exists a $(j,\rho)$-equivariant contracting map from $\H^2$ to $\H^2$. Since this is an open condition on the pair $(j,\rho)$, closed anti-de Sitter $3$-manifolds have a rich deformation theory, which is completely described in my thesis \cite[Chapter 4]{TholozanThese} (see also \cite{Tholozan3}).

Using Kassel's description of closed anti-de Sitter $3$-manifolds and with an explicit differential geometric computation, I proved in \cite{Tholozan5} the following formula for their volume:
\begin{equation}\label{eq: Volume AdS}
\Vol\left((j,\rho)(\pi_1(S))\backslash \PSL(2,\R)\right) = \frac{\pi^2}{2}(\euler(j) + \euler(\rho))~,
\end{equation}
where $\euler$ denotes the Euler class of a representation. In particular, the volume is constant along continuous deformations of the anti-de Sitter structure.

Upon hearing about this result, Labourie got the intuition that it could be derived from Chern--Simons theory, which lead him to present an alternative proof at an MSRI seminar \cite{LabourieMSRI}. Labourie's proof was never published, but he explained it to me in details, which it sparked the present work.

Let us also mention that another proof of \eqref{eq: Volume AdS} was given by Alessandrini and Li \cite{AlessandriniLi15} using Higgs bundles. Finally, I show in \cite{Tholozan5} that the local rigidity of the volume is not true anymore when one considers finite volume non-compact quotients of $\AdS^3$.

\subsubsection{Volume of compact quotients}

Generalizing my work on the volume of anti-de Sitter $3$-manifolds, I inverstigated more systematically in \cite{Tholozan6} the volume of \emph{compact quotients} of reductive symmetric spaces $G/H$, i.e. quotients of $G/H$ by a discrete subgroup $\Gamma$ of $G$ acting freely, properly discontinuously and cocompactly. There, I proved the following formula:

\begin{equation}\label{eq: Volume compact quotients}
\Vol(\Gamma \backslash G/H) = \int_{[\Gamma]} \iota^*\omega_{G,H}~,
\end{equation}
where $\iota: \Gamma \to G$ is the inclusion, $\omega_{G/H}$ is a class in $\HH^\bullet_c(G,\C)$ depending only on $H$, and $[\Gamma]$ is a certain \emph{fundamental class} in $\HH_\bullet(\Gamma, \Z)$.

In many cases, the class $\omega_{G,H}$ happens  to be a characteristic class, and one deduces the rationality of $\Vol(\Gamma \backslash G/H)$ (and, in particular, its local rigidity). Equation \ref{eq: Volume compact quotients} also allows to readily apply Corollary \ref{c: Main Corollary Cohomology} to prove the volume rigidity of compact quotients of $G/H$.

Corollary \ref{c: Main Corollary Volume} is more general in that it deals with any closed manifold locally modelled on $G/H$, which might not a priori be \emph{complete} (i.e. a quotient of the model). While it is conjectured that every closed manifold locally modelled on a reductive homogeneous space is complete (a variation on Markus' conjecture), this conjecture is far from being solved, and it is not even known in general whether completeness is stable under small deformations. It is thus more satisfying to have a volume rigidity result without any completeness assumption.

\subsection{Cohomology of symmetric spaces}

As we just saw, there are many situations where Theorem \ref{t: Main Theorem} was known to follow either from classical arguments in Chern--Weil or Chern--Simons theory. The core of this work will thus be to prove that these arguments in fact account for all the $G$-invariant forms on $G/H$, i.e. that $\Omega^\bullet_\inv(G/H)$ is generated by \emph{Chern--Weil forms} and \emph{Chern--Simons forms} (which will be properly defined in Section \ref{s: Invariant forms on symmetric spaces}). This will be done by reinterpreting old results of Cartan, Chevalley and  Borel on the cohomology of compact symmetric spaces in light of Chern--Simons theory. \\

Recall that the algebra of $G$-invariant forms on $G/H$ is canonically isomorphic to the cohomology algebra of the \emph{dual compact symmetric space} $G_U/H_U$ (see Section \ref{prop: Comparison cohomology compact dual}). We first establish the following version of Cartan--Borel's structure theorem for this cohomology:
 
\begin{MonThm} \label{t: Cohomology symmetric space}
Let $G/H$ be a compact symmetric space. We have
\[\HH^\bullet(G/H, \C) = \HH^\bullet_\even (G/H, \C) \otimes \HH^\bullet_\odd(G/H,\C)~,\]
where
\begin{itemize}
\item the subalgebra $\HH^\bullet_\even (G/H, \C)$ is the algebra of characteristic classes of the \emph{tautological principal $H$-bundle} $G \to G/H$;

\item the subalgebra $\HH^\bullet_\odd(G/H, \C)$ is the pull-back of  $\HH^\bullet(G,\C)$ under the map
\begin{equation} \label{eq: square map}
\begin{array}{cccc}
\iota_{G,H}: & G/H & \to & G\\
& gH & \mapsto & g \sigma(g)^{-1}~.
\end{array}
\end{equation}
\end{itemize}
\end{MonThm}

The map $\iota_{G,H}$ defined in \eqref{eq: square map} will play a crucial role throughout the paper.

\begin{rmk}
While the above theorem is essentially due Cartan and Borel, we do not know if the interpretation of $\HH^\bullet_\odd(G/H,\C)$ as a pull-back under the map $\iota_{G,H}$ was explicitly stated before.
\end{rmk}

To understand the odd part of the cohomology of $G/H$ it is thus enough to understand the cohomology of a compact semisimple Lie group $G$. The cohomology of $G$ does not consist of Chern--Weil characteristic classes, but (perhaps unsurprisingly to experts), we will prove that it is generated by \emph{Chern--Simons classes}. More precisely, consider the space $X_G= G$ equipped with the action of $G\times G$ by left and right multiplication, and let $P_G$ be the principal $G$-bundle
\[G\times G \to X_G~.\]
The bundle $P_G$ carries two invariant flat connections $\Theta_L$ and $\Theta_R$, corresponding respectively to left and right parallelism on $G$. We will prove the following:

\begin{MonLem} \label{l: Cohomology Lie group Chern-Simons}
The algebra $\HH^\bullet(G,\C)$ is generated by the \emph{Chern--Simons classes} associated to the pair of connections $(\Theta_L, \Theta_R)$ on the principal $G$-bundle $P_G$.
\end{MonLem}

\subsection{Assumptions on $G$}

It is easy to find counterexamples to Theorem~\ref{t: Main Theorem} if one removes the assumption that $G$ is semisimple. For instance, let $\mathbb E^2 \equaldef \O(2)\ltimes \R^2 / \O(2)$ be the Euclidean plane and $\omega$ its translation invariant area form. Let $M$ be the $2$-torus $\R^2/\Z^2$. Consider the family of $\mathbb E^2$-structures on $M$ given by
\[\function{\dev_t}{\tilde M = \R^2}{\E^2 = \R^2}{(x,y)}{(tx,ty)}\]
and
\[\function {\rho_t}{\pi_1(M)= \Z^2}{\R^2 \subset \Isom(\E^2)}{(u,v)}{(tu,tv)}~.\]
Then one easily sees that
\[\int_M \dev_t^* \omega = t^2~,\]
which is thus not constant in $t$.

This example can of course be broadly generalized. Let $G/H$ be a homogeneous space such that $G$ admits a non trivial normalizer $N(G)$ in $\Diff(G/H)$. Then $N(G)$ acts on $\Omega_\inv^\bullet(G/H,\C)$. Let $\omega$ be a $G$-invariant form on $G/H$ which is not fixed by $N(G)$. Then conjugating a representation $\rho:\pi_1(M) \to G$ by $N(G)$ will typically change the class $\rho^* \omega$.\\

\subsection{Structure of the paper}

Though many of the results contained here will perhaps be unsurprising to experts, I could not find references that embrace precisely what I need of Chern--Simons theory and cohomology of symmetric spaces. The paper with thus try to be as self-contained as possible.

In Section \ref{s: Chern--Weil and Chern--Simons}, after recalling some background, we introduce Chern--Weil and Chern--Simons forms associated to connections on a principal bundle, and their relations with characteristic classes. Section \ref{s: Invariant forms on symmetric spaces} is devoted to the description of the algebra $\Omega^\bullet_\inv(G/H)$. Relying on the work of Cartan and Borel, we prove Theorems \ref{t: Cohomology symmetric space} and Lemma \ref{l: Cohomology Lie group Chern-Simons}, and deduce that $\Omega^\bullet_\inv(G/H,\C)$ is generated by Chern--Weil and Chern--Simons forms. Finally in Section \ref{s:Local rigidity}, we recall the classical rigidity results for Chern--Weil and Chern--Simons classes, leading to the proof of Theorem \ref{t: Main Theorem}, and its corollaries. 

\subsection*{Acknoledgements}
This paper is very much indebted to Fran\c cois Labourie, who explained to me the nuts and bolts of Chern--Simons theory and how it could be used to prove volume rigidity of locally homogeneous manifolds. I thank him more generally for the inspiration that his work has been for my research.

\section{Chern--Weil and Chern--Simons classes} \label{s: Chern--Weil and Chern--Simons}

Recall that $G$ is a real connected semisimple Lie group. Let $\g$ denote the Lie algebra of $G$. 

\subsection{Principal bundles}

A \emph{principal $G$-bundle} over a manifold $M$ is a smooth fiber bundle
\[p:P\to M\]
equiped with a smooth right action of $G$ preserving the fibers and acting simply transitively on each fiber.

If $V$ is a manifold equipped with a smooth left action of $G$ (for instance, a linear representation of $G$ or a $G$-homogeneous manifold), one can associate to any principal bundle $P$ a fiber bundle 
\[P\times_G V = P \times V/\langle (p,v)\sim (pg, g^{-1}v), g\in G\rangle\]
called the \emph{associated $V$-bundle}. In particular, the \emph{adjoint bundle} of a principal bundle $P$ is the vector bundle
\[\Ad(P) = P\times_G \g\]
associated to the adjoint representation of $G$.

Let $P$ be a principal $G$-bundle over a manifold $M$ and $H$ a Lie subgroup of $G$. The right quotient $P/H$ is canonically isomorphic to the associated bundle $P\times_G (G/H)$. A \emph{reduction of structure group to $H$} of $P$ is a principal $H$-bundle $P'$ equipped with a bundle map $\phi: P'\to P$ that commutes with the right $H$-action. The image of $\phi$ contains a unique right $H$-orbit in each fiber of $P$ and thus factors to a section of the associated bundle $P/H$. Conversely, the preimage in $P$ of a section of $P/H$ is a reduction of structure group of $P$ to $H$.

A \emph{gauge transformation} of a principal bundle $\pi: P\to M$ is a bundle automorphism (i.e. a diffeomorphism $h:P\to P$ such that $\pi\circ h = \pi$) that commutes with the right action of $G$. The group of gauge transformations is canonically identified with the group of smooth sections of the associated bundle 
\[\Aut(P) \equaldef P\times_G G~,\]
where $G$ acts on itself by conjugation.

\subsection{Connections, curvature}
Every $u\in \g$ defines a vector field on $P$ by taking the derivative of the $G$-action, that we denote by $X_u$. The vector fields $X_u$ are tangent to the fibers of the bundle, and the map $u\mapsto X_u$ defines a trivialization of the subbundle $T F$ of $TP$ tangent to the fibers.

\begin{defi}
A \emph{(principal) connection} on $P$ is a $1$-form $\Theta$ with values in $\g$, satisfying the following properties:
\begin{itemize}
\item $\Theta(X_u) =u$ for all $u\in \g$
\item $g^* \Theta = \Ad_g \circ \Theta$ for all $g\in G$.
\end{itemize}
\end{defi}

The difference between two principal connections $\Theta$ and $\Theta'$ is a $1$-form with values in $\g$ satisfying the following properties:
\begin{itemize}
\item $(\Theta- \Theta')(X_u) = 0$ for all $u\in \h$,
\item $g^*(\Theta-\Theta') = \Ad_g\circ (\Theta-\Theta')$ for all $g\in G$. 
\end{itemize}
It thus factors to a $1$-form on $M$ with values in $\Ad(P)$. Hence the space $\Conn(P)$ of principal connections on $P$ is an affine space over the space $\Omega^1(M,\Ad(P))$.

The kernel of a principal connection is a distribution transverse to $TF$ and defines an Ehresmann connection with holonomy in $G$. The $2$-form
\[R_\Theta = \d \Theta + \frac12 [\Theta, \Theta]\]
with values in $\g$ has the following properties:
\begin{itemize}
\item $R_\Theta(X_u, \cdot) = 0$ for all $u\in \g$,
\item $g^* R_\Theta = \Ad_g \circ R_\Theta$ for all $g\in G$.
\end{itemize}
It follows that this form factors to a $2$-form on $M$ with values in $\Ad(P)$ called the \emph{curvature form} of $\Theta$, and that we still denote $R_\Theta$. 


Recall that, given a smooth map $f$ from a manifold $N$ to $M$, one can pull-back a principal bundle $P\to M$ by setting
\[f^*P = \{(x,p)\in N\times P\mid \pi(p)= f(x)\}~.\]
If $P$ is equipped with a principal connection $\Theta$, then $f^*\Theta$ is a principal connection on $f^*P$ with curvature form $f^*R_\Theta$.

\subsection{Flat bundles} \label{ss: Flat bundles}

The curvature form $R_\Theta$ vanishes if and only if the distribution $\ker(\Theta)$ is integrable, in which case the connection is called \emph{flat}. The connection $\Theta$ is flat if and only if $P$ locally admits \emph{parallel sections}, i.e. section $s$ such that $s^*\Theta \equiv 0$.

Given $\rho$ a representation of $\pi_1(M)$ into $G$, define 
\[P_\rho = \tilde M \times G/ \langle (x,g)\sim (\gamma\cdot x, \rho(\gamma) g), \gamma \in \pi_1(M) \rangle~.\]
Then $P_\rho$ is a principal $G$-bundle and the ``trivial'' flat connection on $\tilde M \times G$ factors to a flat connection $\Theta_\rho$ on $P_\rho$.

Conversely if $P$ is a principal bundle equipped with a flat connection $\Theta$, local parallel sections globalize over the universal cover, and one deduces that $(P,\Theta)$ is isomorphic to $(P_\rho,\Theta_\rho)$ for a representation $\rho$ called the \emph{holonomy} of the flat connection. Actually, the holonomy is only defined up to conjugation in $G$ (corresponding to the choice of a trivialisation over the universal cover).

Given a left action of $G$ on a manifold $V$, we have a canonical isomorphism
\[P_\rho \times_G V \simeq  M\times_\rho V \equaldef \tilde M \times V/ \langle (x,v)\sim (\gamma\cdot x, \rho(\gamma) g), \gamma \in \pi_1(M) \rangle~.\]
If $s$ is a section of the fiber bundle $M\times_\rho V \to M$, the lift $\tilde s$ of $s$ to $\tilde M$ is a map from $\tilde M$ to $V$ which is \emph{$\rho$-equivariant}, i.e. satisfies
\[\tilde s(\gamma \cdot x) = \rho(\gamma) \cdot \tilde s(x)\]
for all $\gamma \in \pi_1(M)$. Conversely, any $\rho$-equivariant map $\tilde s: \tilde M \to V$ factors to a section $s$ of $M\times_\rho V$. If $\omega$ is a $G$-invariant form on $V$, then the form $\tilde s^* \omega$ on $\tilde M$ is $\pi_1(M)$-invariant and thus factors to a form on $M$ denoted $s^*\omega$.

\subsection{Characteristic classes and the Chern--Weil homomorphism}

Let $EG$ be a contractible CW complex equipped with a free and proper right action of the topological group $G$. The quotient space $BG$ is called a \emph{classifying space} for the (topological) group $G$ and $EG$ a \emph{universal principal $G$-bundle}.

It is universal in the sense that, given any principal $G$-bundle $P\to M$, there exists a continuous map $f_P:M \to BG$, such that $P$ is isomorphic to the pull-back of the bundle $EG$ by $f_P$. Moreover, $f_P$ is unique up to homotopy. In particular, any two classifying spaces are homotopy equivalent.

\begin{defi}
Let $P\to M$ be a principal $G$-bundle. The \emph{algebra of characteristic classes} of $P$ is the image of the homomorphism
\[f_P^*: \HH^\bullet(BG,\C) \to \HH^\bullet(M,\C)~.\]
\end{defi}
Here, the cohomology can a priori be taken with coefficients in an arbitrary domain. 

Let $\Sym^\bullet_\inv(\g^\vee)$ denote the graded algebra of $\C$-valued polynomials on $\g$ invariant under the adjoint action, with grading given by twice the degree. We will always identify homogeneosu polynomials of degree $k$ with symmetric $k$-linear forms. 

Let $P\to M$ be a principal $G$-bundle and $\Theta$ a connection on $P$. For any homogeneous polynomial $f$ in $\Sym^k_\inv(\g^\vee)$, the form
\[\CW_f(\Theta) \equaldef f(R_\Theta)\]
is a well-defined $2k$-form on $M$ with coefficients in $\C$, which we call a \emph{Chern--Weil form}.

Note that Chern--Weil forms are natural with respect to pull-backs: if $P\to M$ is a principal bundle equipped with a connection $\Theta$ and $\phi:N\to M$ is a smooth map, then 
\[\CW_f(\phi^*\Theta) = \phi^* \CW_f(\Theta)~.\]

One can show that $\CW_f(\Theta)$ is closed, and that its de Rham cohomology class $\cw_f(P)$ does not depend on the connection. We thus get a homomorphism of graded algebras
\begin{equation} \label{eq: Morphism de Chern--Weil}
\function{\Phi_P}{\Sym^\bullet_\inv(\g^\vee)}{\HH^\bullet(M,\C)}{f}{\cw_f(P)}~.
\end{equation}

\begin{theo}[Chern--Weil]
There exists a homomorphism of graded algebras
\[\Phi_{EG}: \Sym^\bullet_\inv(\g^\vee) \to \HH^\bullet(BG,\C)\]
such that, for any principal bundle $P\to M$, the following diagram commutes:
\[
\xymatrix{
\Sym^\bullet_\inv(\g^\vee) \ar[r]^{\Phi_{EG}} \ar[dr]_{\Phi_P} & \HH^\bullet(BG,\C) \ar[d]^{f_P^*} \\ 
& \HH^\bullet(M,\C)~.
}
\]
Moreover, when $G$ is compact, $\Phi_{EG}$ is an isomorphism.
\end{theo}

\begin{rmk}
The notation $\Phi_{EG}$ is meaningful here: $\Phi_{EG}$ is formally the homomorphism of \eqref{eq: Morphism de Chern--Weil} of the universal principal bundle, and it is indeed constructed as an inductive limit of $\Phi_{P_n}$ for principal bundles ${P_n}$ ``approaching''~$EG$.
\end{rmk}

\begin{rmk}
If $G$ is not compact, let $K$ be a maximal compact subgroup of $G$. Then $BK = EG/K$ is a classifying space for $K$ and the fibration $BK\to BG$ is a homotopy equivalence since its fibers $G/K$ are contractible. We thus get that $\HH^\bullet(BG,\C)\simeq \HH^\bullet(BK,\C)$. Over a manifold $M$, this translates into the fact that every principal $G$-bundle $P$ admits a reduction of structure group $P'$ to $K$, unique up to homotopy, and the characteristic classes of $P$ are those of $P'$.
\end{rmk}

The integral structure on the cohomology of $BG$ can be transported to $\Sym^\bullet_\inv(\g^\vee)$:
\begin{defi} \label{defi: Integral polynomial}
A polynomial $\Sym^\bullet_\inv(\g^\vee)$ will be called \emph{integral} (resp. rational) if $\Phi_{EG}(f)$ belongs to the image of $\HH^\bullet(BG,\Z)$ (resp. $\HH^\bullet(BG,\Q)$) in $\HH^\bullet(BG,\C)$.
\end{defi}

If $f$ is integral, then for every principal bundle $P\to M$ the Chern--Weil class $c_f(P)$ belongs to $\HH^\bullet(M,\Z)$. More generally, let $c_1,\ldots, c_n$ be a basis of $\HH^\bullet(BG,\Z)$. Then, for any $f\in \\Sym^\bullet_\inv (\g^\vee)$, we can write 
\[\Phi_{EG}(f) = \sum_{i=1}^n \alpha_i c_i~,\quad \alpha_i\in \C~,\]
and get that, for any principal bundle $P\to M$, the Chern--Weil class $\cw_f(P)$ belongs to the cohomology with coefficients in the submodule $\Vect_\Z(\alpha_1, \ldots, \alpha_n)\subset \C$. This gives strong rational properties to Chern--Weil classes, which live a priori in the de Rham cohomology.

\subsection{Chern--Simons forms}

We recall without proofs the general setting of Chern--Simons theory and refer to the initial paper of Chern--Simons \cite{ChernSimons74} for details.\\

Fix a principal $G$-bundle $P$ over a manifold $M$ and a  polynomial $f\in \Sym^k_\inv(\g^\vee)$. Given two connections $\Theta_0, \Theta_1$ on $P$, the difference between the Chern--Weil forms
\[\CW_f(P,\Theta_1)- \CW_f(P,\Theta_0)\]
is exact. Chern--Simons theory provides a way to construct a primitive to this form, which is well defined up to an \emph{exact} term.

Recall that the space $\Conn(P)$ of principal connections on $P$ is an affine space over $\Omega^1(M,\Ad(P))$.
Given $\Theta\in \Conn(P)$ and $\Psi \in T_\Theta \Conn(P) = \Omega^1(M,\Ad(P))$, we set 
\[A_f(\Psi) = k f(\Psi, R_\Theta, \ldots, R_\Theta)~.\]

We see the map $A_f$ as a $1$-form on $\Conn(P)$ with values in $\Omega^{2k-1}(M,\C)$. We will denote by $\mathrm D$ the exterior derivative on $\Conn(P)$, to avoid confusion with the exterior derivative on $M$. In particular, $\mathrm D A_f$ is a $2$-form on $\Conn(P)$ with values in $\Omega^{2k-1}(M,\C)$, while
\[\d A_f: \Psi \mapsto \d (A_f(\Psi))\]
is a $1$-form on $\Conn(P)$ with values in $\Omega^{2k}(M,\C)$.

With a bit of familiarity with differential calculus on a vector bundle with connection, one can prove the following formulae:
\begin{equation} \label{eq: small d of the Chern--Simons action}
\d A_f(\Psi) = k f(\d \Psi + [\Theta, \Psi], R_\Theta, \ldots, R_\Theta) = \dt_{\vert t= 0} \CW_f(\Theta + t \Psi)~.
\end{equation}

\begin{equation} \label{eq: capital D of the Chern--Simons action}
\mathrm D A_f(\Psi_1, \Psi_2) = -k (k-1) \d f(\Psi_1, \Psi_2, \Omega_\Theta, \ldots, \Omega_\Theta)~.
\end{equation}

The identity \eqref{eq: small d of the Chern--Simons action} can be re-written as
\begin{equation} \label{eq: dA = DC}
\d A_f= \mathrm D \CW_f~,
\end{equation}
where $\CW_f$ is seen as a function on $\Conn(P)$ with values in $\Omega^{2k}(M,\C)$.

The identity \eqref{eq: capital D of the Chern--Simons action} implies that $\mathrm D A_f$ has coefficients in the space of exact forms on $M$.

\begin{defi}
Given a homogeneous $G$-invariant polynomial $f$ on $\g$, we define the \emph{Chern--Simons form} of a piecewise smooth path $(\Theta_t)_{t\in [0,1]}$ in $\Conn(P)$ by
\[\CS_f((\Theta_t)_{t\in [0,1]}) = \int_{t=0}^1 A_f(\dot \Theta_t) \d t~.\]

In particular we define the Chern--Simons form of a pair of connections  $(\Theta_0, \Theta_1)\in \Conn(P)^2$ as the Chern--Simons form of the straight path between them:
\[\CS_f(\Theta_0, \Theta_1) = \CS_f(((1-t)\Theta_0 + t \Theta_1)_{t\in [0,1]})~.\]
\end{defi}

Like Chern--Weil forms, Chern--Simons forms are natural with respect to pull-backs: if $P\to M$ is a principal bundle equiped with a pair of connections $\Theta_0, \Theta_1$ and $\phi:N\to M$ is a smooth map, then 
\[\CS_f(\phi^*\Theta_0, \phi^*\Theta_1) = \phi^* \CS_f(\Theta_0,\Theta_1)~.\\ \]

From \eqref{eq: dA = DC}, we get that the Chern--Simons form is a primitive of the difference between the Chern--Weil forms:

\begin{prop}
For any piecewise smooth path $(\Theta_t)_{t\in [0,1]}$ in $\Conn(P)$, we have
\[\d \CS_f((\Theta_t)_{t\in [0,1]}) = \CW_f(\Theta_1)- \CW_f(\Theta_0)~.\]
\end{prop}

This already shows that the difference between the Chern--Simons forms of two paths with the same endpoints is a closed form. In fact, since $\mathrm D A_f$ takes values into the space of exact forms by \eqref{eq: capital D of the Chern--Simons action}, the Stokes formula (in the infinite dimensional space $\Conn(P)$) shows that this difference is exact.

\begin{prop} \label{prop: Chern-Simons loop exact}
For any piecewise smooth path $(\Theta_t)_{t\in [0,1]}$ in $\Conn(P)$, the form
\[\CS_f((\Theta_t)_{t\in [0,1]}) - \CS_f(\Theta_0,\Theta_1)\]
is exact.

In particular, for any $(\Theta_0, \Theta_1, \Theta_2) \in \Conn(P)^3$, the form
\[\CS_f(\Theta_0, \Theta_1) + \CS_f(\Theta_1, \Theta_2) - \CS_f(\Theta_0,\Theta_2)\]
is exact.
\end{prop}

If, for various reasons, the form $\CS_f(\Theta_0,\Theta_1)$ is closed, then it defines a de Rham cohomology class on $M$ that we denote $\cs_f(\Theta_0,\Theta_1)$. By Proposition~\ref{prop: Chern-Simons loop exact}, the Chern--Simons form of any smooth path from $\Theta_0$ to $\Theta_1$ gives a representative of this class.\\

Let us mention three situations where Chern--Simons forms lead to cohomological invariants.

\subsubsection{Gauge transformations}

Assume that $\Theta_1= h^* \Theta_0$ for some gauge transformation $h$ of $P$. Then $R_{\Theta_1} = \Ad_h^{-1} \circ R_{\Theta_0}$ and, since $f$ is $G$-invariant, we get that $\CW_f(\Theta_1)= \CW_f(\Theta_0)$. Hence $\CS_f(\Theta_0,h^*\Theta_0)$ is closed.

If there exists a smooth path $(h_t)_{t\in [0,1]}$ in the gauge group of $P$ such that $h_0 = \Id_P$ and $h_1 = h$, then one can verify that
\[\cs_f(\Theta_0, h^*\Theta_0) = [\CS_f(h_t^*\Theta_0)_{t\in [0,1]}] = 0~.\]
thus Chern--Simons classes can help distinguish connected components in the gauge group.

In fact, one can interpretate the Chern--Simons class $\cs_f(\Theta_0, h^*\Theta_0)$ as a Chern--Weil class on $M\times \S^1$. Indeed, define a principal bundle $P_h$ over $M\times \S^1$ by 
\[P_h = P\times [0,1] / (p,0) \simeq (h (p), 1)~.\]
Let $p_1 : M\times \S^1 \to M$ denote the projection on the first factor and let $[\d t]$ denote the pull-back to $\HH^1(M\times \S^1)$ of the generator of $\HH^1(\S^1, \Z)$.

\begin{prop} 
For every $f\in \Sym^\bullet_\inv(\g^\vee)$ and every connection $\Theta$ on~$P$, we have
\[p_1^* \cs_f(\Theta, h^* \Theta) \wedge [\d t ] = \cw_f(P_h)~.\]
\end{prop}

\begin{coro} \label{coro: Integral Chern-Simons gauge transformation}
 If $f$ is integral, then $\cs_f(\Theta, h^*\Theta)$ belongs to the image of $\HH^\bullet(M,\Z)$ in $\HH^\bullet(M,\C)$.
\end{coro}

\subsubsection{Flat connections}
When $\Theta_0$ and $\Theta_1$ are flat, $\CW_f(\Theta_0) = \CW_f(\Theta_1) = 0$, hence $\CS_f(\Theta_0,\Theta_1)$ is closed. The following immediate proposition is key in proving local rigidity of Chern--Simons invariants:

\begin{prop} \label{prop: Vanishing CS flat family}
Let $(\Theta_t)_{t\in [0,1]}$ be a smooth path of flat connections on $P$. Then 
\[\cs_f(\Theta_0, \Theta_1) = 0~.\]
\end{prop}

\begin{proof}
By Proposition \ref{prop: Chern-Simons loop exact}, we have
\[\cs_f(\Theta_0, \Theta_1) = [\CS_f((\Theta_t)_{t\in [0,1]})]~.\]
From the definition, we see that $\CS_f((\Theta_t)_{t\in [0,1]}) = 0$ since $R_{\Theta_t} = 0$ for all~$t$.
\end{proof}

Assume $f$ is integral. Let $\rho_0$ and $\rho_1$ be two representations of $\pi_1(M)$ into $G$ such that the associated principal bundles $P_{\rho_0}$ and $P_{\rho_1}$ are isomorphic. Then there exist two flat connections $\Theta_0$ and $\Theta_1$ on the same principal bundle with respective holonomies $\rho_1$ and $\rho_2$. Moreover, $\Theta_0$ and $\Theta_1$ are uniquely defined up to a gauge transformation. By Corollary \ref{coro: Integral Chern-Simons gauge transformation} and Proposition~\ref{prop: Chern-Simons loop exact}, we thus get a well-defined cohomology class
\[\cs_f(\rho_0,\rho_1) \equaldef \cs_f(\Theta_0, \Theta_1) \mod \Z~.\]

Thus, while Chern--Weil classes can distinguish connected components in $\Hom(\pi_1(M),G)$ by distinguishing the homeomorphism type of the associated principal bundles, Chern--Simons classes can distinguish between connected components of representations whose associated principal bundles are isomorphic.

\subsubsection{Chern--Simons theory on $3$-manifolds}

Though this paragraph is not useful to the rest of the paper, we thought that including it would clarify how our presentation of Chern--Simons theory relates to its extensive developments in three-dimensional topology.\\

Any simple Lie algebra admits an invariant bilinear symmetric form 
\[\kappa: (u,v)\mapsto \Tr(\ad_u \ad_v)\]
called the \emph{Killing form}. Moreover there is a constant $a\neq 0$ such that $a\kappa$ is integral in the sense of Definition \ref{defi: Integral polynomial}.

Assume moreover that $G$ is connected and simply connected. Because $\pi_2(G) =\{0\}$, every principal $G$-bundle over a $3$-manifold is trivial and thus carries a flat connection with trivial holonomy, that we denote by $D$.

Now let $P$ be a principal $G$-bundle over a closed $3$-manifold $M$ with a connection $\Theta$. The form 
\[\CS_{a\kappa} (D,\Theta)\]
is trivially closed since it is a $3$-form. Moreover, its value modulo $\Z$ does not depend on the choice of the trivialization $D$ since $a\kappa$ is integral. Integrating over $M$, one can thus associate a Chern--Simons invariant
\[\mathfrak{cs}(\Theta) \equaldef \int_M  \CS_{a\kappa} (D,\Theta) \in \C/\Z\]
to any principal $G$-bundle with a connection $\Theta$. One can apply this for instance to the Levi--Civita connection of a Riemannian metric to define the Chern--Simons invariant of a closed Riemannian $3$-manifold.

\section{The algebra of invariant forms on symmetric spaces} \label{s: Invariant forms on symmetric spaces}

Recall that we have fixed a connected semisimple Lie group $G$ and an involutive isomorphism $\sigma$ of $G$, and denoted by $H$ the neutral component of the subgroup fixed by $\sigma$. Up to quotienting $G$, we can assume without loss of generality that $H$ does not contain a non-trivial normal subgroup of $G$, so that $G$ acts faithfully on the symmetric space $G/H$. The involution $\sigma$ induces an involution of $G/H$ which fixes the base point $o \equaldef \1_G H$ and acts as $-\Id$ on $T_o G/H$.

In this section, we describe the algebra $\Omega^\bullet_\inv(G/H,\C)$ of $G$-invariant forms on $G/H$, relying mostly on the work of Cartan \cite{Cartan50} and Borel \cite{Borel53}. Our goal is to prove that this algebra is generated by Chern--Weil forms and Chern--Simons forms associated to $G$-invariant connections on automorphic bundles over $G/H$.

\subsection{Invariant forms on symmetric spaces}

The study of invariant differential forms on symmetric spaces started with \'Elie Cartan and was carried on by his son Henri. A first elementary but useful result is the following:

\begin{prop}[E. Cartan]
Every $G$-invariant form on $G/H$ is closed.
\end{prop}

In other words, the complex of $G$-invariant forms $\Omega^\bullet_\inv(G/H,\C)$ has trivial differential and is thus isomorphic to its cohomology. A series of isomorphisms identifies it canonically with the cohomology of the \emph{dual compact symmetric space}.

Define first the complexification $\g_\C$, $\h_C$, $G_\C$ and $H_\C$ of $\g$, $\h$, $G$ and $H$ respectively in the following way:
\begin{itemize}
\item $\g_\C$ is the complex Lie algebra $\g\otimes_\R \C$,
\item $\h_\C$ is the complex Lie subalgebra $\h \otimes_\R \C$,
\item $G_\C$ is the neutral component of $\Aut(\g_\C)$,
\item $H_\C$ is the neutral component of the subgroup preserving $\h_\C$.
\end{itemize}

Note that, because $\g_\C$ is a semisimple Lie algebra, it is the Lie algebra of the complex group $G_\C$. Since $\h_\C$ is the subalgebra fixed by an involution, it is the Lie algebra of its stabilizer.\footnote{While the complexification is perfectly well-defined at the level of Lie algebras, there is something arbitrary in our definition of the complexification of Lie group. For instance, the complexification of $\SL(k,\R)$ is the adjoint group $\PSL(k,\C)$. We did not settle with a more algebraic definition because we want to consider connected simple Lie groups such as $\SO_\circ (p,1)$, which are not necessarily algebraic.}

Recall that a Cartan involution of an algebraic group $G_\C$ is an involutive isomorphism whose set of fixed points is compact. By a result of E. Cartan (see also \cite{Richardson68}), we can always choose a Cartan involution $\theta$ of $G_\C$ that commutes with $\sigma$. Moreover, it is unique up to conjugation by $H_\C$. Define $G_U$ to be subgroup of $G_\C$ fixed by $\theta$, $H_U = G_U\cap \H_C$, and $\g_U$ and $\h_U$ their respective Lie algebras. The groups $G_U$ and $H_U$ are compact real forms of $G_\C$ and $H_\C$. In particular, they are connected (since the complex groups are connected and retract to their maximal compact subgroup). We call $G_U/H_U$ the \emph{dual compact symmetric space}.

given a linear representation $V$ of a group $G$ over a field $k$ and $l$ an extension of $k$, denote by $\Lambda^\bullet_l(V^\vee)^G$ the subalgebra of $G$-invariant forms in the algebra of alternate $k$-multilinear forms.

\begin{prop} \label{prop: Comparison cohomology compact dual}
There are canonical isomorphisms
\[\begin{array}{ccccc}
\Omega^\bullet_\inv(G/H,\C) &\simeq & \Lambda^\bullet_\C((\g/\h)^\vee)^H &\simeq &\Lambda^\bullet_\C((\g_C/\h_C)^\vee)^{H_\C} \\
&\simeq & \Lambda^\bullet_\C((\g_U/\h_U)^\vee)^{H_U} &\simeq & \Omega^\bullet_\inv(G_U/H_U,\C) \\ & & &\simeq & \HH^\bullet(G_U/H_U,\C)~.
\end{array}
\]
\end{prop}

\begin{proof}
The first isomorphism is well-known: every $G$-invariant form on $G/H$ restricts to an alternate form on $\g/\h \simeq T_o G/H$ which is invariant under the adjoint action of $H$, and conversely, every such alternate form extends in a unique way to a $G$-invariant form on $G/H$.

For the second isomorphism, note first that every $H$-invariant $\R$-multilinear form $\alpha$ on $\g/\h$ extends uniquely to an $H$-invariant $\C$-multilinear form $\alpha_\C$ on $\g_\C/\h_\C = \g/\h \otimes_\R \C$. Since $\alpha_\C$ is $H$-invariant, it satisfies $u\cdot \alpha_\C \equaldef \dt_{\vert t=0} \exp(-t\ad_u)^* \alpha_\C = 0$ for every $u\in \h$. But since $\g_\C$ is a complex Lie algebra and $\alpha_\C$ is $\C$-multilinear, the map 
\[u\mapsto u \cdot \alpha_\C\]
is $\C$-linear. Hence the relation $u\cdot \alpha_\C$ is satisfied for all $u \in \h_\C$. Finally, integrating this relation gives
\[h^*\alpha_\C = \alpha_\C\]
for all $h\in H_\C$ since $H_\C$ is connected by definition. Hence $\alpha \mapsto \alpha_\C$ gives the second isomorphism.

The third and fourth isomorphisms are respectively the second and first one applied to the compact symmetric space $G_U/H_U$, which has the same complexification. Finally, on $G_U/H_U$, every closed form is cohomologous to a unique invariant one obtained by averaging under the action of $G_U$. Hence the natural map
\[\Omega^\bullet_\inv(G_U/H_U,\C) \to \HH^\bullet(G_U/H_U,\C)\]
is an isomorphism.
\end{proof}

The space $G$ equipped with the right action of $H$ is a principal bundle over $G/H$ which we call the \emph{tautological $H$-bundle}. The involution $\sigma$ preserves the Killing form of $G$, and since $\h$ is the eigenspace of the involution for the eigenvalue $1$, the Killing form is non-degenerate in restriction to $\h$.
\begin{defi}
The \emph{standard principal connection} on the bundle $G\to G/H$ is the left $G$-invariant form $\Theta_{G/H}$ on $G$ with values in $\h$ whose value at $\1_G$ is the orthogonal projection to $\h$ (for the Killing form). We denote by $R_{G/H}$ its curvature form.
\end{defi}

For any $f \in \Sym^\bullet_\inv(\h^\vee)$, the Chern--Weil form 
\[\CW_f(\Theta_{G/H})\]
thus defines a $G$-invariant form on $G/H$. We set
\[\Omega^\bullet_{\even}(G/H,\C) = \{\CW_f(\Theta_{G/H}), f\in  \Sym^\bullet_\inv(\h^\vee)\}~.\]

Looking at the restriction of the Chern--Weil forms to $\g/\h$ and their behaviour under complexification, one proves the following:

\begin{prop}
The sequence of isomorphisms given in Proposition \ref{prop: Comparison cohomology compact dual} restrict to an isomorphism
\[\Omega^\bullet_{\even}(G/H,\C) \simeq \HH^\bullet_{\even}(G_U/H_U,\C)~,\]
where $\HH^\bullet_{\even}(G_U/H_U,\C)$ is the algebra of characteristic classes of the tautological principal bundle of $G_U/H_U$.
\end{prop}

Cartan and Borel gave the following description of the cohomology of compact symmetric spaces:
\begin{CiteThm}[Cartan, Borel, \cite{Borel53}] \label{thm: Cartan Borel}
Let $G_U/H_U$ be a compact symmetric space, with $G_U$ and $H_U$ connected. Then there exists a subalgebra 
\[\HH^\bullet_{\odd}(G_U/H_U,\C)\subset \HH^\bullet(G_U/H_U,\C)~,\]
generated by forms of odd degree, such that
\[\HH^\bullet(G_U/H_U,\C) = \HH^\bullet_\even(G_U/U,\C) \otimes \HH^\bullet_\odd(U/U,\C)~.\]
Moreover, the pull-back map
\[p^*: \HH^\bullet_\odd(U/U,\C)\to \HH^\bullet (U,\C)\]
induced by the projection $p:U\to G_U/H_U$ vanishes on $\HH^{>0}_\even(G_U/H_U,\C)$ and is injective on $\HH^\bullet_\odd(G_U/H_U,\C)$.
\end{CiteThm}

While the subalgebra $\HH\bullet_\odd(G_U/H_U,\C)$ is not uniquely determined, Theorem \ref{thm: Odd forms G/H} below will give us a canonical way to construct it and to define a corresponding subalgebra $\Omega^\bullet_\odd(G/H,\C) \subset \Omega^\bullet_\inv(G/H,\C)$.

\subsection{Bi-invariant forms on Lie groups}

A particular class of symmetric spaces will be of interest to us: the semisimple Lie group $G$ itself equipped with the action of $G\times G$ given by \[(g,h)\cdot x = gxh^{-1}~.\]  
The corresponding involution $\sigma$ of $G\times G$ is given by $\sigma(g,h) = (h,g)$, the stabilizer of the base point $\1_G$ is the diagonal subgroup $\Delta(G)$, and the central symmetry at $\1_G$ is the map $g\mapsto g^{-1}$. Since it will be convenient to distinguish between $G$ seen as a group and $G$ seen as a symmetric space, we will denote the latter by $X_G$.

The dual compact symmetric space of $X_G$ is simply the symmetric space $X_{G_U} = G_U\times G_U /\Delta(G_U)$. In particular, Proposition \ref{prop: Comparison cohomology compact dual} gives an isomorphism
\[\Omega^\bullet_\inv(X_G,\C) \simeq \HH^\bullet(G_U,\C)~.\]

The projection map $G_U \times G_U \to X_{G_U}$ induces in cohomology a \emph{coproduct} \[\delta: \HH^\bullet(G_U,\C) \to \HH^\bullet(G_U,\C) \otimes \HH^\bullet(G_U,\C)~,\]
giving $\HH^\bullet(G_U,\C)$ the structure of a \emph{Hopf algebra}. Following Borel, we define:

\begin{defi}
A class $x \in \HH^k(G_U,\C)$ is \emph{primitive} if 
\[\delta(x) = x\otimes \1 + (-1)^{k} \1\otimes x~.\]
\end{defi}

We denote by $\Prim(G_U)$ the vector space spanned by primitive classes in $\HH^\bullet(G_U,\C)$, and by $\Prim(X_G)$ the corresponding subspace of $\Omega^\bullet_\inv(X_G)$. The general structure theorem of Hopf gives the following:

\begin{CiteThm}[Hopf] \label{thm: Hopf}
The inclusion $\Prim(G_U)\hookrightarrow \HH^\bullet(G_U,\C)$ induces an isomorphism
\[\Lambda^\bullet \Prim(G_U) \simeq \HH^\bullet(G_U,\C)~,\]
where $\Lambda^\bullet$ denotes the exterior algebra.
\end{CiteThm}
Consequently, we also get that
\[\Omega^\bullet_\inv(X_G,\C) = \Lambda^\bullet\Prim(X_G)~.\\ \]

The space of primitive forms is further described by the work of Chevalley. Let $\mu\in \Omega^1(G,\g)$ denote \emph{Maurer--Cartan form} of $G$, i.e. the left-invariant $1$-form which is the identity at $\1_G$. Let $J\subset \Omega^\bullet_\inv(X_G,\C)$ denote the square of the ideal of forms of positive degree, i.e. the ideal generated by forms that can be factored as a product of two forms of lower degree. The following theorem is attributed by Borel to Cartan, Chevalley and Weil.

\begin{CiteThm}[Cartan--Chevalley--Weil]
Let $\tau: \Sym^{>0}_\inv(\g^\vee) \to \Omega^\bullet_\inv(G,\C)$ be the linear map sending a symmetric $k$-linear form $f$ to \[f(\alpha, [\alpha,\alpha], \ldots, [\alpha, \alpha])~.\]
Then the kernel of $\tau$ is the ideal $J$ and the image of $\tau$ is the set of primitive forms $\Prim(X_G)$.
\end{CiteThm}

\begin{rmk}
Chevalley also proved that $\Sym^\bullet_\inv(\g^\vee)$ is a polynomial algebra generated by $\rank_\C(G)$ elements. One deduces that
\[\dim \Sym^{>0}_\inv(\g^\vee)/J = \dim \Prim(X_G) = \rank_\C(G)~.\]
\end{rmk}

The formal similarity between the Cartan--Chevalley--Weil description of primitive forms and the definition of the Chern--Simons invariants is more than a coincidence. Let $p_1$ and $p_2$ denote the projections of $G\times G$ to the first and second factor and define
\[\Theta_L = p_2^* \mu~, \quad \Theta_R=  p_1^*\mu~,\]
where $\mu$ is the Maurer--Cartan form. Then $\Theta_R$ and $\Theta_L$ are two flat invariant connections on the tautological bundle $G\times G\to X_G$. The associated adjoint bundle is canonically identified to the tangent bundle to $G$, and the associated connections $\nabla^L$ and $\nabla^R$ respectively make the left-invariant and right-invariant vector fields parallel.\footnote{Note that both connections are indeed \emph{bi-invariant}, because the right-translate of a left-invariant vector field is another left-invariant vector field.}

\begin{theo} \label{thm: Primitive forms -> Chern-Simons}
For every $f \in \Sym^k_\inv(\g^\vee)$, we have
\[\CS_f(\Theta_L, \Theta_R) = \frac{(-1)^k k! ((k-1)!}{2^{k-1}(2k-1)!} \tau(f)~.\]
\end{theo}

Since the image of $\tau$ generates $\Omega^\bullet_\inv(X_G,\C)$, we deduce: 
\begin{coro}
The algebra of invariant forms on $X_G$ is generated by the Chern--Simons forms associated to the pair of connections $(\Theta_L,\Theta_R)$ on the tautological principal bundle.
\end{coro}

\begin{proof}[Proof of Theorem \ref{thm: Primitive forms -> Chern-Simons}]
Set $\Theta_t= (1-t) \Theta_L + t \Theta_R$ and denote by $R_t = \d \Theta_t + \frac{1}{2}[\Theta_t, \Theta_t]$ its curvature, seen as a $2$-form on $G\times G$ with values in $\g$. Since $\Theta_L$ and $\Theta_R$ are flat, one computes that
\[R_t = \frac{-t(1-t)}{2} \left( [\Theta_R, \Theta_R] + [\Theta_L, \Theta_L]\right) + t (1-t)[\Theta_L,\Theta_R]~.\]

Define $\hat \CS_f\in \Omega^2(G\times G, \g)$ by
\begin{equation} \label{eq: Computing Chern--Simons}
\hat \CS_f = k \int_{t=0}^1 f(\dot \Theta_t, R_t, \ldots , R_t)~.
\end{equation}
By construction, $\hat \CS_f$ is the pull-back to $G\times G$ of the Chern--Weil form $\CS_f(\Theta_L, \Theta_R)$, hence$\CS_f(\Theta_L, \Theta_R)$ is the pull-back of $\hat \CS_f$ by any section of the bundle. Consider the section $s: g \mapsto (\1_G, g)$. Then $s^*\Theta_R = 0$ and $s^* \Theta_L= \mu$, hence
\[s^* \dot \Theta_t = - \mu\]
and 
\[s^*R_t = \frac{-t(1-t)}{2} [\mu, \mu]~.\]
We can thus compute:
\begin{eqnarray*}
\CS_f(\Theta_L, \Theta_R) &=& k \int_{t=0}^1 f(s^* \dot \Theta_t, s^* R_t, \ldots, R_t)\d t\\
&=& (-1)^k k \left (\int_{t=0}^1 \left(\frac{t(1-t)}{2}\right)^{k-1} \d t\right)\, f(\alpha, [\alpha,\alpha], \ldots, [\alpha, \alpha])\\
&=& \frac{(-1)^k k! ((k-1)!}{2^{k-1}(2k-1)!} \tau(f)~.
\end{eqnarray*}
 
\end{proof}

\subsection{The inclusion $\iota_\sigma$}

Let us recall the definition of the map $\iota_{G,H}$ given in the introduction:

\begin{defi}
The map $\iota_{G,H}: G/H \to X_G$ is defined by
\[\iota_{G,H}(gH) = g \sigma(g)^{-1}\]
for every class $gH$ in $G/H$. This does not depend on the representative $g$ since $h\sigma(h)^{-1} = \1_G$ for all $h\in H$.
\end{defi} 

The map $\iota_{G,H}$ is an inclusion of symmetric spaces. It is equivariant with respect to the representation $(\Id, \sigma): G\to G\times G$. In particular, every invariant form on $X_G$ can be pulled-back by $\iota_{G,H}$ to a $G$-invariant form on $G/H$. Moreover, this operation commutes with the comparison isomorphisms, i.e. we have the following diagram:
\[
\xymatrix{
\Omega^\bullet_\inv(X_G,\C) \ar[r]^{\iota_{G,H}^*} \ar[d]_\simeq & \Omega^\bullet_\inv(G/H,\C) \ar[d]^\simeq \\
\HH^\bullet(G_U,\C) \ar[r]^{\iota_{G_U,H_U}^*} & \HH^\bullet(G_U/H_U,\C)~.
}
\]

We can now state our refined version of Cartan--Borel's theorem for $G$-invariant forms on a (not necessarily compact) symmetric space:

\begin{theo} \label{thm: Odd forms G/H}
The space $\Omega^\bullet_\inv(G/H,\C)$ is isomorphic to 
\[\Omega^\bullet_\even(G/H,\C) \otimes \Omega^\bullet_\odd(G/H,\C)~,\]
where $\Omega^\bullet_\even(G/H,\C)$ is the algebra of Chern--Weil forms of the standard connection on the tautological principal $H$-bundle and 
\[\Omega^\bullet_\odd(G/H,\C) = \iota_{G,H}^* \Omega^\bullet_\inv(G,\C)~.\]
\end{theo}
Note that this is just a reformulation of Theorem \ref{t: Cohomology symmetric space} from the introduction in the general case where $G/H$ is not assumed compact. The two theorems are equivalent thanks to the comparison isomorphisms.\\

To prove the theorem, we will use the following cute and elementary lemma of linear algebra:
\begin{lem} \label{lem: Pair of projectors}
Let $E$ be a Euclidean vector space, $F_1$ and $F_2$ any two subspaces, and $\pi_1$ and $\pi_2$ the respective orthogonal projections to $F_1$ and $F_2$. Then 
\[F_1 = \im {\pi_1}_{\vert F_2} \oplus \ker {\pi_2}_{\vert F_1}~.\]
\end{lem}

\begin{proof}
Let us first compute dimensions: we have
\begin{eqnarray*}
\dim\im {\pi_1}_{\vert F_2} &=& \dim F_2 - \dim \ker F_2 \cap F_1^\perp\\
&=& \dim F_2 - (\dim E - \dim (F_2^\perp + F_1)) + \dim F_1 \cap F_2^\perp\\
&=& \dim F_2 - \dim E + \dim F_2^\perp +\dim F_1 - \dim F_1 \cap F_2^\perp \\
&=& \dim F_2 - \dim \ker {\pi_2}_{\vert F_1} ~.
\end{eqnarray*}

It is thus enough to prove that $\im {\pi_1}_{\vert F_2} \cap \ker {\pi_2}_{\vert F_1} = \{0\}$.

Let $v$ be a vector in $\im {\pi_1}_{\vert F_2} \cap \ker {\pi_2}_{\vert F_1}$, and write $v= \pi_1(w)$, $w\in F_2$. Then
\begin{eqnarray*}
\langle v,v\rangle &=&\langle \pi_1(w), w\rangle \quad \textrm{since $w - \pi_1(w) \in F_1^\perp$}\\
&=& \langle \pi_2 \circ \pi_1(w), w\rangle \quad \textrm{since $w \in F_2$}\\
&=& \langle \pi_2(v), w \rangle = 0~.
\end{eqnarray*}
Hence $v= 0$ since $E$ is euclidean.
\end{proof}

\begin{proof}[Proof of Theorem \ref{thm: Odd forms G/H}]
Thanks to the comparison isomorphisms, it is enough to prove the theorem for compact groups, for which invariant forms are in bijection with cohomology classes. The fact that the cohomology of $G/H$ is the tensor product of an even part and an odd part is already given by Cartan and Borel, and the only thing we need to prove is that one can choose 
\[\HH^\bullet_\odd(G/H,\C) = \iota_{G,H}^*\HH^\bullet(G,\C)\]
in their theorem.

Let $I$ denote the ideal in $\HH^\bullet(G/H,\C)$ generated by $\HH^{>0}_\even (G/H,\C)$. Since Theorem \ref{thm: Cartan Borel} already tells us that $\HH^\bullet(G/H,\C)$ is a tensor product, what we need to prove is mainly that $\iota_{G,H}^* \HH^\bullet(G,\C)$ is a complement to $I$.

Fix an integer $k$, and let $E$ be the space of $H$-invariant alternate $k$-forms on $\g$, which we equip with an $H$-invariant scalar product. Consider the following subspaces of $E$:
\begin{itemize}
\item the subspace $F_1$ of forms that are pulled back from an $H$-invariant form on $\g/\h$,
\item the subspace $F_2$ of $G$-invariant forms.
\end{itemize}

Recall that $F_1$ identifies naturally with $\HH^k(G/H,\C)$, while $F_2$ identifies naturally with $\HH^k(G,\C)$. Through this identification, we see the pull-back maps $\pi^*: \HH^k(G/H,\C) \to \HH^k(G,\C)$ and $\iota_{G,H}^*: \HH^k(G,\C)\to \HH^k(G/H,\C)$ as maps between $F_1$ and $F_2$.

Let $\pi_1$ and $\pi_2$ denote respectively the orthogonal projections to $F_1$ and~$F_2$. We claim that we have the following identifications :
\begin{equation}\label{eq: cohomology map from G/H to G}
\pi^*= {\pi_2}_{\vert F_1}~,
\end{equation}
\begin{equation}\label{eq: cohomology map from G to G/H}
\iota_{G,H}^*= 2^k {\pi_1}_{\vert F_2}
\end{equation}

\begin{itemize}
\item {\it Proof of \eqref{eq: cohomology map from G/H to G}:} Let $\omega$ be a $G$-invariant $k$-form on $G/H$ and $\omega'$ its pull-back to $G$, which is a closed left-invariant $k$-form on $G$. It is cohomologous to a unique bi-invariant form $\overline{\omega'}$ on $G$, obtained by averaging $\omega'$ under the right action of $G$ (with respect to the Haar probability measure on $G$).

Since $\omega'$ is already left invariant, the restriction of $\overline{\omega'}$ to $T_{\1_G} G = \g$ is given by
\[\overline{\omega'}_{\1_G} = \int_G \Ad_g^* \omega'_{\1_G} \d g~.\]
Since averaging under the adjoint $G$-action is the orthogonal projection to $F_2$, we conclude that 
\[\pi^*[\omega] = \pi_2([\omega])\]
after the appropriate identifications.\\

\item {\it Proof of \eqref{eq: cohomology map from G to G/H}}:
Denote by $p$ the projection $\g \to \g/\h$ and $j : \g/\h \to \g$ be the section of $p$ with image $\h^\perp$. Since the derivative of the involution $\sigma$ is the identity on $\h$ and minus the identity on $\h^\perp$, we get that 
\[\d_o \iota_{G,H} = 2 j~,\]
where $o$ denotes the basepoint $\1_G H$.

Note that $j \circ p: \g \to \g$ is the orthogonal projection to $\h^\perp$, from which one easily deduces that $p^*j^* : E \to F_1$ is the orthogonal projection $\pi_1$ on $F_1$.

Let now $\omega$ be a biinvariant form on $G$ and $\omega'$ its pull-back to $G/H$. Then we have
\[p^* (\omega'_o) = 2^k p^*j^*\omega_{\1_G} = 2^k \pi_1(\omega_{\1_G})~.\]
which rewrites
\[\iota_{G,H}^*[\omega] = 2^k \pi_1([\omega])\]
after the apropriate identifications.
\end{itemize}

\noindent {\it Conclusion of the proof:} Applying Lemma \ref{lem: Pair of projectors}, we get that
\[\HH^\bullet(G/H,\C) = \ker \pi^* \oplus \iota_{G,H}^* \HH^\bullet(G,\C)~.\]
By Theorem \ref{thm: Cartan Borel}, the kernel of $\pi^*$ is the ideal $I$ generated by $\HH^{>0}_\even(G/H,\C)$. Hence $\HH^\bullet(G/H,\C)$ is generated by $\HH^\bullet_\even(G/H,\C)$ and $\iota_{G,H}^*(G,\C)$. 

The rest is a completely general argument: let $\HH^\bullet_\odd(G/H,\C)$ be a subalgebra such that 
\[\HH^\bullet(G/H,\C) = \HH^\bullet_\even(G/H,\C) \otimes \HH^\bullet_\odd(G/H,\C)~.\]
quotienting by the ideal $I$ we get that $\HH^\bullet_\odd(G/H,\C) \simeq \iota_{G,H}^* \HH^\bullet(G,\C)$. The inclusions of $\HH^\bullet_\even(G/H,\C)$ and $\iota_{G,H}^*\HH^\bullet(G,\C)$ induce a morphism
\[\HH^\bullet_\even(G/H,\C) \otimes \iota_{G,H}^*\HH^\bullet(G,\C) \to \HH^\bullet(G/H,\C)\simeq  \HH^\bullet_\even(G/H,\C) \otimes \HH^\bullet_\odd(G/H,\C)\]
which is surjective. Since both algebras have the same dimension over $\C$, it is an isomorphism.
\end{proof}

\section{Local rigidity of cohomological invariants}
 \label{s:Local rigidity}
 
This section will conclude the proof of Theorem \ref{t: Main Theorem} and its corollaries.
 
\subsection{Smooth families of principal bundles} \label{ss: Families of bundles}

Let us first introduce properly the notion of \emph{smooth family of bundles} and smooth family of sections used informally in the introduction. Here, the term ``bundle'' is meant to include principal bundles as well as their associated bundles.

We call $(E_t)_{t\in [0,1]}$ a \emph{smooth family of bundles} on a manifold $M$ if there is a smooth bundle $E$ over $M\times [0,1]$ such that $E_t$ is the pull-back of $E$ under the map $x\mapsto (x,t)$. If $(E_t)_{t\in [0,1]}$ is a smooth family of principal $G$-bundles on $M$, we call $(s_t)_{t\in [0,1]}$ a \emph{smooth family of sections} of $E_t$ is each $s_t$ is the pull-back under $x\mapsto (x,t)$ of a smooth section $s$ on $E$.

The following classical lemma can be attributed to Ehresmann:
\begin{lem} \label{lemma: family of principal bundles}
Let $M$ be a smooth manifold and $E$ a bundle over $M\times [0,1]$. Then $E$ is isomorphic to $p_1^* E_0$, where $p_1: M\times [0,1] \to M$ is the projection to the first factor.
\end{lem}
More informally, any smooth family of vector bundles is trivial. As a consequence we have the following:
\begin{coro} \label{coro: family of sections}
Let $(E_t)_{t\in [0,1]}$ be a smooth family of bundles over $M$ and $s$ a smooth section of $E_0$. Then there exists a smooth family of sections $(s_t)_{t\in [0,1]}$ of $(E_t)_{t\in [0,1]}$ such that $s_0=s$. Moreover, if $(s'_t)_{t\in [0,1]}$ is another such family, then $s'_t$ is isotopic to $s_t$ for all $t$.
\end{coro}

The main example we will be interested in is smooth families of flat bundles. Recall that, fixing a finite generating set $S$ of $\pi_1(M)$, one can see $\Hom(\pi_1(M),G)$ as an analytic subset of the analytic manifold $G^S$. We call a path $(\rho_t)_{t\in [0,1]}$ in $\Hom(\pi_1(M),G)$ smooth if it $t\mapsto \rho_t$ is a smooth map from $[0,1]$ to $G^S$ (equivalently, if $t\mapsto \rho_t(\gamma)$ is smooth in $G$ for all $t$). The space $\Hom(\pi_1(M),G)$ is locally connected by smooth arcs, in the sense that any two representations $\rho_0$ and $\rho_1$ in the same connected component are the endpoints of a smooth path in $\Hom(\pi_1(M),G)$.

Let $V$ be a manifold equipped with a $G$-action. If $(\rho_t)_{t\in [0,1]}$ is a smooth family of representations then the family
\[\left(M\times_{\rho_t} V\right)_{t\in [0,1]}\]
is a smooth family of bundles. In particular, $(P_{\rho_t})_{t\in [0,1]}$ is a smooth family of principal bundles. By Lemma \ref{lemma: family of principal bundles}, this family is trivial, and we deduce:
\begin{coro}
Let $(\rho_t)_{t\in [0,1]}$ be a smooth family of representations of $\pi_1(M)$ into $G$. Then there exists a principal $G$-bundle $P$ and a smooth family of flat connections $(\Theta_t)_{t\in [0,1]}$ on $P$ such that $\Theta_t$ has holonomy $\rho_t$.

Moreover, if $P'$ is another principal $G$-bundle and $(\Theta'_t)_{t\in [0,1]}$ another family of flat connections with monodromies $(\rho_t)_{t\in [0,1]}$, then there exists a smooth family $(\phi_t)_{t\in [0,1]}$ of bundle isomorphisms from $P$ to $P'$ such that $\Theta'_t = \phi_t^* \Theta_t$.
\end{coro}

In particular, the topology of the bundle $P_\rho$ is constant when $\rho$ varies in a connected component of $\Hom(\pi_1(M),G)$.

\subsection{Proof of Theorem \ref{t: Main Theorem}}

We now have a sufficient understanding of invariant forms on symmetric spaces to prove Theorem \ref{t: Main Theorem}. Clearly, it is enough to prove it for $\omega$ in a subset of $\Omega^\bullet_\inv(G/H,\C)$ generating the whole algebra. We will thus prove it separately for $\omega
\in \Omega^\bullet_\even(G/H,\C)$ (using Chern--Weil theory) and for $\omega \in \iota_{G,H}^*\Prim(X_G)$ (using Chern--Simons theory). While both arguments are probably well-known it is worth including them here for completeness.

\subsubsection{The even case}
Let us first reinterpret the definition of $\Omega^\bullet_\even(G/H,\C)$ after pull-back by a section of a flat $G/H$ bundle. Fix $\omega \in\Omega^\bullet_\even(G/H,\C)$ and let $f \in \Sym^\bullet_\inv(\h^\vee)$ be such that
\[\omega = f(R_{G/H})~,\]
where $R_{G/H}$ is the curvature of the standard connection on the tautological principal $H$-bundle $G\to G/H$.

\begin{prop} \label{prop: Pull-back Chern-Weil}
For any manifold $M$, any representation $\rho: \pi_1(M) \to G$ and any smooth section $s$ of the flat bundle $M\times_\rho (G/H)$, we have
\[[s^*\omega] = \cw_f (P_\rho(s))~,\]
where $P_\rho(s)$ is the reduction to $H$ of the flat $G$-bundle $P_\rho$ given by $s$.
\end{prop}

\begin{proof}
Denote by $P_{G/H}$ the tautological principal $H$-bundle over $G/H$ and let $\tilde s$ be the lift of $s$ to $\tilde M$. Then $\pi_1(M)$ acts on $\tilde s^* P_{G/H}$ via $\rho$, and, by construction, the quotient $\pi_1(M) \backslash \tilde s^* P_{G/H}$ is the principal $H$-bundle $P_\rho(s)$.

Now, the standard connection $\Theta_{G/H}$ on $P_{G/H}$ pulls back to a $\pi_1(M)$-invariant connection on $\tilde s^* P_{G/H}$ which factors to a connection $s^* \Theta_{G/H}$ on $P_\rho(s)$, with curvature $s^* R_{G/H}$. By naturality of the pull-back, we have
\begin{eqnarray*}
\cw_f (P_\rho(s)) &=& [\CW_f(s^* \Theta_{G/H})]\\
&=& [f(s^*R_{G/H})]\\
&=& [s^*\omega]~.
\end{eqnarray*}
\end{proof}

\begin{coro} \label{coro: Rigidity Chern Weil}
For any manifold $M$, any smooth path $(\rho_t)_{t\in [0,1]}$ in $\Hom(\pi_1(M),G)$ and any smooth family of sections $s_t$ of $P_{\rho_t}$, we have
\[ [s_t^*\omega] = [s_0^*\omega]\]
for all $t$.
\end{coro}

\begin{proof}
The principal $H$-bundles $P_{\rho_t}(s_t)$ form a smooth family of bundles over $M$. By Lemma \ref{lemma: family of principal bundles}, they are all isomorphic and, in particular, they have the same characteristic classes. Hence
\[ [s_t^*\omega] = \cw_f(P_{\rho_t}(s_t)) = \cw_f(P_{\rho_0}(s_0)) = [s_0^* \omega]\]
for all $t\in [0,1]$.
\end{proof}

\subsubsection{The odd case}
Now, let $\rho_1$, $\rho_2$ be two representations of $\pi_1(M)$ into $G$ and $s$ a section of the flat bundle $M\times_{(\rho_1,\rho_2)} X_G$. Denote by $(P_{\rho_i}, \Theta_{\rho_i})$ the flat principal bundles associated to $\rho_i$, $i=1,2$. Then $s$ induces an isomorphism of principal bundles 
\[\function{\phi_s}{P_{\rho_2}}{P_{\rho_1}}{(x,g)}{(x, \tilde s (x) g)}~.\]

Choose $\omega\in \Prim(X_G)$ and let $f\in \Sym^k_\inv(\g^\vee)$ be such that $\omega = \frac{(-1)^k k! ((k-1)!}{2^{k-1}(2k-1)!} \tau(f)$.

\begin{prop} \label{prop: Pull-back Chern--Simons}
For any manifold $M$, any pair of representations $(\rho_1, \rho_2)\in \Hom(\pi_1(M),G)^2$ and any smooth section $s$ of the flat $X_G$-bundle associated to $(\rho_1,\rho_2)$, we have
\[ [s^*\omega] = \cs_f(\Theta_{\rho_1}, \phi_s^*\Theta_{\rho_2})~.\]
In particular, for any representation $\rho \in \Hom(\pi_1(M),G)$ and any smooth section $s$ of the corresponding flat $G/H$-bundle, we have
\[ [s^* (\iota_{G,H}^* \omega)] = \cs_f(\Theta_\rho, \phi_{\iota_{G,H}\circ s}^* \Theta_{\sigma \circ \rho})~.\]
\end{prop}

\begin{proof}
Denote by $P_{X_G}$ the tautological principal $G$-bundle over $X_G$ and $\Theta_L$ and $\Theta_R$ the two invariant connections on $P_{X_G}$ given respectively by left and right parallelism. 

Let $\tilde s$ be the lift of $s$ to $\tilde M$. Then $\pi_1(M)$ acts on $\tilde s^* P_{X_G}$ via $(\rho_1,\rho_2)$ and, by construction, the quotient $\pi_1(M) \backslash \tilde s^* P_{X_G}$ equipped with the pulled-back connections $s^* \Theta_L$ and $s^*\Theta_R$ is isomorphic to the principal $G$-bundle $P_{\rho_1}$ equipped with the flat connections $\Theta_{\rho_1}$ and $\phi_s^*\Theta_{\rho_2}$. 

Now, by Theorem \ref{thm: Primitive forms -> Chern-Simons}, we have
\[\omega = \CS_f(\Theta_L, \Theta_R)~,\]
and by naturality under pull-back, we deduce that
\begin{eqnarray*}
[s^*\omega] &=& [s^*\CS_f(\Theta_L, \Theta_R)]\\
&=& [\CS_f(s^*\Theta_L, s^*\Theta_R)]\\
&=& \cs_f(\Theta_{\rho_1}, \phi_s^*\Theta_{\rho_2})~.
\end{eqnarray*}
The second part of the theorem is an immediate consequence of the first part, since $\iota_{G,H} \circ \tilde s$ is $(\rho, \sigma \circ \rho)$-equivariant.
\end{proof}

\begin{coro} \label{coro: rigidity Chern--Simons}
For any manifold $M$, any smooth path $(\rho_{1,t},\rho_{2,t})_{t\in [0,1]}$ in $\Hom(\pi_1(M), G\times G)$ and any smooth family of sections $s_t$ of $M \times_{(\rho_{1,t}, \rho_{2,t})} X_G$, we have
\[ [s_t^*\omega] = [s_0^*\omega]\]
for all $t$.

In particular, for any smooth path $(\rho_t)_{t\in [0,1]}$ and any smooth family of sections $s_t$ of $M\times_{\rho_t} (G/H)$, we have
\[ [s_t^* (\iota_{G,H}^*\omega)] = [s_0^* (\iota_{G,H}^*\omega)]\]
for all $t$.
\end{coro}

\begin{proof}
By Lemma \ref{lemma: family of principal bundles}, there exists a principal $G$-bundle $P$ over $M$ with two smooth families of connections $(\Theta_{1,t})_{t\in [0,1]}$ and $(\Theta_{2,t})_{t\in [0,1]}$ such that 
\[(P,\Theta_{1,t},\Theta_{2,t}) \simeq (P_{\rho_{1,t}}, \Theta_{\rho_{1,t}},\phi_{s_t}^*\Theta_{\rho_{2,t}})\]
for all $t$.

By Proposition \ref{prop: Pull-back Chern--Simons}, we have
\begin{eqnarray*}
[s_t^*\omega] - [s_0^*\omega]&=& \cs_f(\Theta_{1,t},\Theta_{2,t})-\cs_f(\Theta_{1,0},\Theta_{2,0})\\
&=& \cs_f(\Theta_{2,t}, \Theta_{2,0}) - \cs_f(\Theta_{1,t}, \Theta_{1,0})~.
\end{eqnarray*}

Since $\Theta_{2,0}$ and $\Theta_{2,t}$ (resp. $\Theta_{1,0}$ and $\Theta_{1,t}$) are joined by a smooth path of flat connections, we conclude that 
\[[s_t^*\omega] = [s_0^*\omega]\]
by Proposition \ref{prop: Vanishing CS flat family}.
\end{proof}

\subsubsection{Conclusion of the proof}

We have proven Theorem \ref{t: Main Theorem} for $\omega \in \Omega^\bullet_\even(G/H,\C)$ (Corollary \ref{coro: Rigidity Chern Weil}) and for $\omega \in \iota_{G,H}^* \Prim(X_G)$ (Corollary \ref{coro: rigidity Chern--Simons}). By Theorem \ref{thm: Hopf}, $\Prim(X_G)$ generates $\Omega^\bullet_\inv (X_G)$; hence $\iota_{G,H}^* \Prim(X_G)$ generates $\Omega^\bullet_\odd(G/H,\C)$. Since $\Omega^\bullet_\inv(G/H,\C) = \Omega^\bullet_\even(G/H,\C)\otimes \Omega^\bullet_\odd(G/H,\C)$ by Theorem \ref{thm: Odd forms G/H}, we conclude that Theorem \ref{t: Main Theorem} holds for any $G$-invariant form.

\subsection{Continuous group cohomology} \label{ss: Continuous group cohomology Proof }

The \emph{continuous cohomology} $\HH^\bullet_c(G,\C)$ (with constant coefficients) of a Lie group $G$ is the cohomology of the complex $\mathcal C^\bullet_c(G,\C)^G$, where $\mathcal C^k_c(G,\C)^G$ is the space of $G$-invariant continuous functions on $G^{k+1}$, equipped with the usual differential:
\[\d f(g_0, \ldots , g_{k+1}) = \sum_{j=0}^{k+1} (-1)^j f(g_0, \ldots, \hat g_j, \ldots, g_{k+1})~.\]
When $G$ is connected, The \emph{Van Est isomorphism} identifies $\HH^\bullet_c(G,\C)$ with the algebra of $G$-invariant forms on the symmetric space $G/K$. To be more precise, recall that if $M$ is a manifold, there is a map $\pi$ from $M$ to a classifying space of $\pi_1(M)$, unique up to homotopy, which induces the identity on the fundamental groups.

\begin{CiteThm}[Van Est] \label{thm: Van Est}
There is an isomorphism of graded algebras
\[\VE: \HH^\bullet_c(G,\C) \to \Omega^\bullet_\inv(G/K,\C)\]
such that, for any smooth manifold $M$ and any representation $\rho: \pi_1(M) \to G$, the following diagram commutes:
\[
\xymatrix{
\HH^\bullet_c(G,\C) \ar[r]^{\VE} \ar[d]_{\rho^*} & \Omega^\bullet_\inv(G/K,\C) \ar[d]^{\rho^*} \\
\HH^\bullet(\Gamma,\C) \ar[r]^{\pi^*} & \HH^\bullet(M,\C)~.
}
\]
\end{CiteThm}
Here the map $\rho^*:  \HH^\bullet_c(G,\C) \to \HH^\bullet(\Gamma,\C)$ is the pull-back map on group cohomology while $\rho^*: \Omega^\bullet_\inv(G/K,\C) \to \HH^\bullet(M,\C)$ maps $\omega$ to $[s^*\omega]$ for any smooth section of $M\times_\rho (G/K)$.

\begin{proof}[Proof of Corollary \ref{c: Main Corollary Cohomology}]

Let $\Gamma$ be a finitely presented group. Assume first that $\Gamma$ is the fundamental group of an aspherical manifold $M$. Then the map $\pi^*: \HH^\bullet(\Gamma,\C) \to \HH^\bullet(M,\C)$ is an isomoprhism. Let $(\rho_t)_{t\in [0,1]}$ be a smooth family of representations of $\Gamma$ into $G$ and let $\alpha$ be a continuous cohomology class in $\HH^\bullet_c(G,\C)$. Then we have
\begin{eqnarray*}
\pi^* \rho_t^* \alpha &=& \rho_t^* \VE(\alpha) \quad \textrm{by Theorem \ref{thm: Van Est}}\\
&=& \rho_0^* \VE(\alpha) \quad \textrm{by Theorem \ref{t: Main Theorem}}\\
&=& \pi^* \rho_0^* \alpha~,
\end{eqnarray*}
and we conclude that $\rho_t^* \alpha= \rho_0^* \alpha$ for all $t$ since $\pi^*$ is injective.

In general, while $\Gamma$ need not be the fundamental group of an aspherical manifold, one can always find for any $n\in \N$ a manifold $M$ such that $\pi_1(M) = \Gamma$ and $\pi_k(M) = \{0\}$ for all $2\leq k\leq n$. Then $\pi^*: \HH^k(\Gamma,\C) \to \HH^k(M,\C)$ is an isomorphism for all $1\leq k\leq n$. Applying the above arguments thus gives that $\rho_t^* \alpha$ is constant in $t$ for all $\alpha \in \HH^{\leq n}(\Gamma,\C)$. Since $n$ is arbitrary, the conclusion follows.
\end{proof}

\subsection{Volume of locally homogeneous spaces}

In this section we prove Corollary \ref{c: Main Corollary Volume}, namely the volume rigidity of manifolds locally modelled on a reductive homogeneous space $G/H$.\\

The Killing form of $\g$ restricted to $\h^\perp \simeq \g/\h$ extends to a $G$-invariant metric on $G/H$. Moreover, $G/H$ can be oriented since $H$ is assumed to be connected. Hence the metric induces a $G$-invariant volume form $\vol_{G/H}$.

Let $(\dev_t,\rho_t)_{t\in [0,1]}$ be a smooth family of $G/H$-structures on a closed manifold $M$. By definition, the \emph{volume} of the $G/H$-structure $(\dev_t, \rho_t)$ is the number
\[\int_M \dev_t^*\vol_{G/H}~.\]
When $G/H$ is symmetric, Corollary \ref{c: Main Corollary Volume} is thus a direct application of Theorem \ref{t: Main Theorem}, taking $\tilde s_t = \dev_t$ and $\omega = \vol_{G/H}$.

We prove the general case where $G$ is semisimple and $H\subset G$ is reductive by lifting the $G/H$-structure to a $X_G$-structure on an $H/\Lambda$-bundle over $M$.

\begin{proof}[Proof of Corollary \ref{c: Main Corollary Volume}]

Fix a uniform lattice $\Lambda$ in $H$ (it exists by Borel--Harish-Chandra's theorem).

Let $M$ be a closed manifold and $(\dev_t, \rho_t)$ a smooth family of $G/H$-structures on $M$. Let $\tilde P\to \tilde M$ denote the pull-back of the tautological principal $H$-bundle over $G/H$ by $\dev_t$ and and $P\to M$ its quotient under $\pi_1(M)$ (the bundle $P$ does not depend on $t$ by Lemma \ref{lemma: family of principal bundles}).

By construction, the map $\dev_t$ lifts to a local diffeomorphism $\hat \dev_t: \tilde P \to X_G = G$, which satisfies
\[\hat \dev_t(\gamma \cdot p) = \rho_t(\gamma) \hat \dev_t(p)\]
for $\gamma \in \pi_1(M)$ and 
\[\hat \dev_t(p \cdot \lambda) = \hat \dev_t(p) \lambda\]
for all $\lambda \in \Lambda$.

Setting
\[\function{\hat \rho_t}{\pi_1(M)\times \Lambda}{G\times G}{(\gamma,\lambda)}{(\rho_t(\gamma), \lambda)~,}\]
we thus get that $(\hat \dev_t, \hat \rho_t)$ defines a $X_G$-structure on the closed manifold $P/ \Lambda$. Applying Theorem \ref{t: Main Theorem} to the symmetric space $X_G$, we deduce that
\[\int_{P/\Lambda} \hat \dev_t^*\vol_{X_G} = \int_{P/\Lambda} \hat \dev_0^*\vol_{X_G}\]
for all $t$.

On the other hand, after normalizing the volume forms of $X_G$, $H$ and $G/H$ compatibly, we have
\[\int_{P/\Lambda} \hat \dev_t^*\vol_{X_G} = \Vol(H/\Lambda) \int_M \dev_t^*\vol_{G/H}~.\]
We thus conclude that 
\[\int_M \dev_t^*\vol_{G/H} = \int_M\dev_0^* \vol_{G/H}\]
for all $t$.
\end{proof}

\bibliographystyle{plain}
\bibliography{biblio}

\end{document}